\documentclass{amsart}
\usepackage{amsfonts}
\usepackage{amsmath,amssymb}	
\usepackage{amsthm}
\usepackage{amscd}
\usepackage{graphics}
\usepackage{graphicx}
{
\theoremstyle{definition}
\newtheorem{Def}{Definition}

\newtheorem{Rem}{Remark}

}

\newtheorem{Prop}{Proposition}
\newtheorem{Thm}{Theorem}

\begin{document}
\title[cobordism-like modules generic maps of codimension -2 induce]{Structures of cobordism-like modules induced from generic maps of codimension -2}

\author{Naoki Kitazawa}
\address{Institute of Mathematics for Industry, Kyushu University, 744 Motooka, Nishi-ku Fukuoka 819-0395, Japan}
\email{n-kitazawa.imi@kyushu-u.ac.jp}
\maketitle
\begin{abstract}
The {\it Reeb space} of a smooth map whose codimension is minus is the space defined as the space of all connected components of inverse images. For generic maps such as Morse functions and their higher dimensional versions, they are polyhedra whose dimensions are equal to those of the target manifolds and which have simplicial structures compatible with (the canonical) simplicial structures of the source and the target manifolds, and in considerable cases they inherit fundamental and important invariants of source manifolds. In fact, Reeb spaces are fundamentall tools in the algebraic and differential topological theory of generic maps or the global singularity theory.
As one of studies of global topological properties of Reeb spaces, Hiratuka and Saeki showed in 2013 that for generic maps or more precisely, maps compatible with simplicial structures of the manifolds, inducing simplicial structures on the Reeb spaces and having connected components of inverse images of regular values being not null-cobordant, the top-dimensional homology groups with appropriate coefficient rings of the Reeb spaces do not vanish. Later the author extended this theorem: the author has introduced cobordism-like groups based on adjacent relations of connected components of inverse images of regular values and shown a similar theorem.

  In this paper, as a fundamental and new study, we study structures of the cobordism-like modules in cases where the codimensions are -2 for a suitable class of generic maps and show that the structures are flexible.  
\end{abstract}

\section{Introduction}
\label{sec:1}
\subsection{Reeb spaces, generic maps compatible with triangulations of manifolds and their Reeb spaces.}
The {\it Reeb space} of a map is the space defined as the space of all connected components of inverse images. For a continuous map $c:X \rightarrow Y$, we denote the space by $W_c$, and the quotient map by $q_f:X \rightarrow W_c$. We can define the map $\bar{c}$ uniquely defined so that the relation $c=\bar{c} \circ q_c$ holds.
 Reeb spaces are fundamental and important tools in the algebraic and differential topological theory of Morse functions and their higher dimensional versions or the global singularity theory. See also \cite{reeb} and \cite{sharko} for example for the fundamental algebraic, geometric and differential topological theory of Reeb spaces.
 
For example, for generic maps including Morse functions and their higher dimensional versions such as {\it fold} maps etc. and more generally, maps belonging to the class of {\it Thom} maps, existing densely in any space consisting of all smooth maps between smooth manifolds, the Reeb spaces are polyhedra compatible with the canonical PL structures of the source and target manifolds (\cite{shiota}). In addition, {\it topological} stable maps, which are defined as smooth maps such that by slight deformations the topologies of the sets of all singular points, those of singular values and maps themselves are invariant topologically, are also Thom maps and always exist densely, for example, the dimensions of the Reeb spaces are same as those of the target manifolds; 
 moreover, for such maps, for suitable cases, the Reeb spaces inherit fundamental and important invariants such as homology groups etc. of the source manifolds.

Morse functions such that at distinct singular points the values are distinct, existing densely on any closed manifold, are simplest examples of topologically stable maps. They are also {\it stable} maps, class $C^{\infty}$ versions of topologically stable maps: stable maps exist densely for considerable pairs of the dimensions of the manifolds including the cases where the dimensions of the target manifolds are lower than $6$. For fundamental facts and discussions on the singularity theory of generic maps including stable maps, see \cite{golubitskyguillemin} for example.

For presented maps including good maps such as Thom maps, Hiratuka and Saeki (\cite{hiratukasaeki} and \cite{hiratukasaeki2}) have shown the following: if there exists a connected component of an inverse image of a regular value being (oriented) null-cobordant, then the top-dimensional homology group whose coefficient ring is $\mathbb{Z}/2\mathbb{Z}$ (resp. the oriented cobordism group of the dimension equal to the codimension or the difference of the dimensions of the source and the target manifolds) of the Reeb space does not vanish. Later, the author extended this by introducing general cobordism-like modules of manifolds whose dimensions are equal to the codimension observing adjacent relations of inverse images of regular values by such a map and by showing a theorem corresponding to this new module (\cite{kitazawa}).

\subsection{The contents of this paper, terminologies and notation.}

In this paper, as a different, new, fundamental and important study, we investigate a general property on structures of the modules introduced by the author in \cite{kitazawa} for an explicit class of maps of codimension $-2$ and show that the structures are flexible.

This paper is organized as the following.

In section \ref{sec:2}, we review {\it Reeb-triangulable} maps, whose Reeb spaces are polyhedra compatible with the PL structures of the target manifolds and defined first in \cite{kitazawa}. They are generalizations of smooth functions whose Reeb spaces are graphs. Maps compatible with simplicial structures of manifolds, explained in the introduction including Thom maps etc., are called {\it triangulable} maps in \cite{shiota} and later in \cite{hiratukasaeki} and \cite{hiratukasaeki2} and note that triangulable maps are Reeb-triangulable maps and that the converse facts hold in considerable cases. We also introduce {\it fold} maps, higher dimensional versions of Morse functions, and note that fold maps including stable ones are Reeb-triangulable. 
  
In section \ref{sec:3}, we review a module {\it compatible} with a Reeb-triangulable map first introduced in \cite{kitazawa}. It is a submodule of 
 a free module generated by all smooth, closed and connected manifolds whose dimensions are equal to the difference of the dimensions of the source and the target manifolds modulo the equivalence relation that two manifolds are diffeomorphic
 (in the oriented category) and it is a submodule including all elements obtained canonically by observing adjacent relations of inverse images of regular values of the map.

In section \ref{sec:4}, we present and show a main result. For this, methods used in
\cite{masumotosaeki}, \cite{michalak}, \cite{michalak2} etc. are key ingredients or closely related. As a key proposition, we show Proposition \ref{prop:1} and this essentially leads us to complete the proof of the main theorem.

In this paper, manifolds and maps between them are smooth and of class $C^{\infty}$ unless otherwise stated. The {\it singular set} of a smooth map is defined as the set of all singular points of the map, the {\it singular value set} of the map is defined as
 the image of the singular set and the {\it regular value set} of the map is the complement of the singular value set.
  
Throughout this paper, $M$ is a closed manifold of dimension $m \geq 1$, $N$ is a manifold of dimension $n$ without boundary satisfying
  the relation $m > n \geq 1$, and $f$ is a map from $M$ into $N$.
   Moreover, we denote the singular set of the map by $S(f)$.

The author is a member of the project and supported by the project Grant-in-Aid for Scientific Research (S) (17H06128 Principal Investigator: Osamu Saeki)
"Innovative research of geometric topology and singularities of differentiable mappings"
(https://kaken.nii.ac.jp/en/grant/KAKENHI-PROJECT-17H06128/
).
\section{Reeb-triangulable maps and fold maps}
\label{sec:2}

In \cite{kitazawa}, we introduced the following class of generic smooth maps.

\begin{Def}
\label{def:1}
A smooth map $c:X \rightarrow Y$ is said to be {\it Reeb-triangulable} if the Reeb space $W_c$ 
is regarded as a simplicial complex satisfying the following.
\begin{enumerate}
\item The inverse images of the interior of each simplex of dimension $\dim W_c$ of $W_c$ contains no singular point.  
\item The set of all points in $W_c$ whose inverse images have singular points forms a subcomplex of $W_c$.
\item For the map $\bar{c}:W_c \rightarrow Y$, there exists a pair $(\Phi,\phi)$ of homeomorphisms
 where $\Phi$ makes $W_c$ the same simplicial complex as already given one and where $\phi$ makes $Y$ the same PL manifold as already canonically given one and the map ${\phi}^{-1} \circ \bar{c} \circ \Phi$ is a simplicial map.
\item $\dim W_c=\dim Y$.
\end{enumerate}
\end{Def}

Morse functions on closed manifolds and smooth functions on closed manifolds such that the singular sets are finite are Reeb-triangulable. More generally, smooth functions whose Reeb spaces are graphs are Reeb-triangulable.

{\it Fold} maps, which are higher dimensional versions of Morse functions are also Reeb-triangulable and we review the definition of a {\it fold} map. The main theorem is for a certain class of fold maps.
See also \cite{golubitskyguillemin} and for a precise explicit differential topological study, see \cite{saeki} for example.

\begin{Def}
\label{def:2}
A smooth map is said to be a {\it fold} map if at each singular point $p$, the map is of the form
 $$(x_1,\cdots,x_n,\cdots,x_m) \mapsto (x_1,\cdots,x_{n-1},{\sum}_{j=n}^{m-i(p)} {x_j}^2-,{\sum}_{j=m-i(p)+1}^{m} {x_j}^2)$$
for an integer $0 \leq i(p) \leq \frac{m-n+1}{2}$.
\end{Def}
\begin{Prop}
\label{prop:1}
The singular set of a fold map above is an {\rm (}$n-1${\rm )}-dimensional closed submanifold and the restriction map of a fold map to the singular set is an immersion of codimension $1$.
\end{Prop}
Note that a fold map is stable if and only if the restriction map to the singular set is transversal. See also FIGURE \ref{fig:1}.
\begin{figure}
\includegraphics[width=60mm]{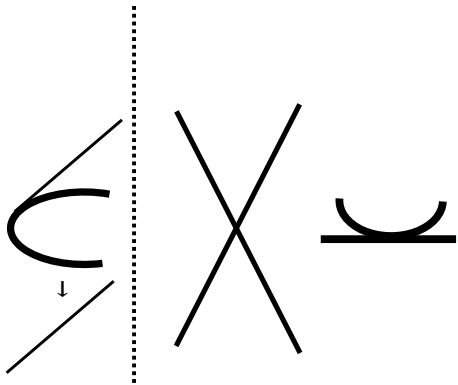}
\caption{A local subset of the singular set of a fold map and local subsets of the singular value sets of fold maps being stable and not being stable, respectively.}
\label{fig:1}
\end{figure}

\begin{Rem}
As a more general study, fold maps are generalized as {\it Morin} maps (\cite{morin}) and stable Morin maps exist densely when the dimension of the source manifold is not smaller than that of the target manifold and the dimension of the target manifold is not so large (smaller than $4$). (Stable) Morin maps are Reeb-triangulable and in the case where the dimension of the target manifold is not so large, stable maps are Reeb-triangulable and exist densely.
 In the proceeding sections, we only treat fold maps which may not be stable.  
\end{Rem}

\section{A cobordism-like module generated by equivalence classes of smooth closed and connected manifolds}
\label{sec:3}

 We introduce a cobordism-like module generated by equivalence classes of smooth, closed and connected manifolds of a fixed dimension, which was first introduced in \cite{kitazawa} and the explanation may overlap with that in the paper in considerable parts. 

We call an equivalence class obtained by considering the equivalence relation on the family of smooth manifolds defined by (orientation preserving) diffeomorphisms between smooth manifolds a {\it diffeomorphism type} (resp. an {\it oriented diffeomorphism type}).

  Let $k \geq 0$ be an integer. For a principle ideal domain $R$, we
   can define a free $R$-module ${\mathcal{N}}_k(R)$ (${\mathcal{O}}_k(R)$) generated by all (resp. oriented) diffeomorphism types of $k$-dimensional smooth, closed and connected (resp. oriented) manifolds so that distinct classes are mutually independent. 

 Let $f:M \rightarrow N$ be a Reeb-triangulable map from an $m$-dimensional closed {\rm (resp. and oriented{\rm )} manifold of dimension $m$ into an $n$-dimensional oriented manifold. We define a submodule $A \subset {\mathcal{N}}_{m-n}(R)$ {\rm (}resp. ${\mathcal{O}}_{m-n}(R)${\rm )} satisfying several conditions.

An arc smoothly embedded in $N$ is said to be {\it transverse} to $f$ if for each point $a$ in the interior and each point $b \in f^{-1}(a)$, the relation $df(T_b M) \oplus T_a S=T_a N$ holds and each point $a$ of the boundary is a regular value
 of $f$ ($f^{-1}(a)$ may be empty). 

Let $\Gamma$ be a triangulation of $N$ given by an appropriate homeomorphism compatible with the PL structure. Let ${\Gamma}_f$ be the set of all connected components of inverse images of all arcs transverse to the map $f$, having at most $1$ point of
 an ($n-1$)-dimensional face of a simplicial complex of $\Gamma$ and satisfying the following by $\bar{f}:W_f \rightarrow N$ satisfying $f=\bar{f} \circ q_f$.
\begin{enumerate}
\item The point in the face is not in any ($n-2$)-dimensional face.
\item The point in the face is in the interior of the curve in $N$.
\end{enumerate}
Let $\gamma \in {\Gamma}_{f}$ and let $A_{f,\gamma}$ be the set of all closed (resp. canonically oriented) connected submanifolds appearing as connected components of the inverse images of points in $\partial \gamma \bigcap {\rm Int} W_f$ by $f$.  The sum of the products of $1$ or $-1$ and (resp. oriented) diffeomorphism types of all submanifolds in $A_{f,\gamma}$ is defined and we denote it by $a_{f,\gamma}$: we set the coefficient of each diffeomorphism type according to the orientation of the ($n-1$)-dimensional face induced from the $n$-dimensional oriented simplex the boundary point of the closed interval in the target belongs to. For example, FIGURE \ref{fig:2} shows a local subspace of the Reeb space of a stable Morse function on an $m$-dimensional closed manifold around a value of a singular point corresponding to a $1$-handle or ($m-1$)-handle and in the figure, $a_1$, $a_2$ and $a_3$ represent the elements corresponding to the corresponding connected components of the inverse images of corresponding regular values (circles represent manifolds appearing as inverse images of regular values and two manifolds in the left and the manifold in the right are cobordant; the latter manifold is regarded as a connected sum of the former two manifolds) and the sum $a_{f,\gamma}$ is $a_1-a_2-a_3$ or $a_2+a_3-a_1$.
%

\begin{figure}
\includegraphics[width=35mm]{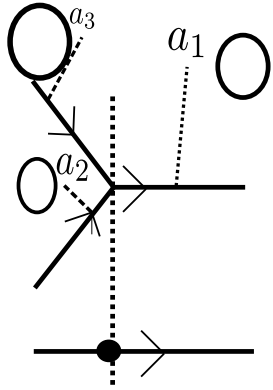}
\caption{The inverse image of a $1$-dimensional polyhedron $\gamma \in {\Gamma}_{f}$ in the Reeb space of a stable Morse function $f$ by the quotient map.}
\label{fig:2}
\end{figure}

\begin{Def}
\label{def:3}
In the discussion just before, if for some triangulation $\Gamma$ of $N$ and each $\gamma \in {\Gamma}_{f}$, $a_{f,\gamma} \in A$ holds, then the submodule $A \subset {\mathcal{N}}_{m-n}(R)$ (resp. ${\mathcal{O}}_{m-n}(R)$) is said to be {\it {\rm (resp.} oriented{\rm )} compatible} with $f$.
\end{Def}
Note that the sign of each $a_{f,\gamma} \in A$ does not affect on the definition.
\begin{Def}
\label{def:3}

Let $R$ be a principle ideal domain. In the discussion
 just before, let $A$ be an $R$-module (resp. oriented) compatible with a Reeb-triangulable map $f$ satisfying the following.

\begin{enumerate}
\item If $A$ contains all elements in ${\mathcal{N}}_{m-n}(R)$ (resp. ${\mathcal{O}}_{m-n}(R)$) corresponding to (resp. oriented) manifolds not appearing as connected components of inverse images
 of regular values, then $A$ is said to be {\it normally compatible} or {\it NC} with $f$.
\item We set the submodule generated by all elements corresponding to manifolds not appearing as connected components of inverse images of regular values as $A_1$. If $A$ is NC with $f$, thus $A_1$ is a submodule of $A$ and for a submodule $A_2$ generated by a finite number of elements represented as linear combinations of elements corresponding to manifolds appearing as connected components of inverse images of regular values, $A$ is represented as the direct sum of $A_1$ and $A_2$, then $A$ is said to be {\it completely normally compatible} or {\it CNC} with $f$. We call $A_1$ the {\it outer part} of $A$ and $A_2$ an {\it effective part} of $A$.
\item Let $A$ be CNC with $f$. An effective part $A_2$ of $A$ is said to be {\it minimal}, if $A_2$ is minimal for any effective part of any submodule CNC with $f$. If $A_2$ is minimal, then $A$ is said to be {\it minimal} as a module CNC with $f$ or $MCNC$ with $f$.  
\end{enumerate}  
\end{Def}

Note that a module MCNC with a Reeb-triangulable map is uniquely determined. In the next section, as a main theorem, conversely, we obtain a map from a given suitable module so that the original module is MCNC with the map.

\section{The main theorem}
\label{sec:4}
\begin{Thm}
\label{thm:1}
Let $R$ be a principle ideal domain. Let $A \subset {\mathcal{O}}_{2}(R)$ be a submodule satisfying the following.
\begin{enumerate}
\item $A$ includes elements corresponding to all oriented diffeomorphism types of closed, connected and oriented surfaces but a finite number of oriented diffeomorphism types and we denote the submodule generated by all of such elements by $A_1$.
\item $A$ includes a finite number of elements represented as linear combinations of oriented diffeomorphism types of closed, connected and oriented surfaces not in $A_1$, the submodule generated by all of such elements are not zero and we denote the submodule by $A_2$.   
\end{enumerate}
Thus, for a closed and connected oriented manifold $N$ of dimension $n$, there exist a closed and connected manifold $M$ of dimension $m=n+2$ and a map $f:M \rightarrow N$ such that $A$ is oriented compatible and MCNC with $f$. Moreover, $A_1$ is the outer part and $A_2$ is the minimal effective part.
\end{Thm}

The following is important in the proof of Theorem \ref{thm:1}. The method used in the proof is similar to one used in \cite{michalak}.

\begin{Prop}
\label{prop:2}
Let ${\partial}_1 P$ and ${\partial}_2 P$ be closed and oriented surfaces which may not be connected. Then there exist a $3$-dimensional, compact, connected and orientable manifold $P$ whose boundary is the disjoint union of ${\partial}_1 P$ and ${\partial}_2 P$ and a function on $P$ satisfying the following.
\begin{enumerate}
\item The function is not constant.
\item The function is Morse.
\item ${\partial}_1 P$ coincides with the inverse image of the minimal value of the function and  ${\partial}_2 P$ coincides with the inverse image of the maximal value of the function.
\item There exists just one singular value of the function and it is not the minimal value or the maximal one.  
\item The inverse image of the singular value by the function is connected.
\end{enumerate}
\end{Prop}

\begin{proof}
The proof is essentially based on a fundamental fact on the differential topological theory of Morse functions, which relates a singular point to an attachment of a handle to a compact manifold represented as the inverse image of the set of all real numbers not larger than a suitable value close to and smaller than the singular value and conversely, an attachment of a handle to a singular point. Such an argument appears in \cite{michalak} in the case where the boundary is a disjoint union of standard spheres of a general dimension.

To a closed and oriented surface ${\partial_1} P$ which may not be connected and which is regarded as ${\partial_1} P \times \{0\} \subset {\partial_1} P \times [-1,0]$, by attaching handles whose indice are $1$ or $2$ being mutually disjoint as presented in FIGURE \ref{fig:3} simultaneously suitably, we can change the surface into a new surface ${\partial_2} P$. More precisely, we add $1$-handles and $2$-handles to change ${\partial_1} P$ into $S^2$, decompose it into finite copies of $S^2$ (the number is equal to that of connected components of a desired surface) and then attach $1$-handles to each component being diffeomorphic to $S^2$: we can demonstrate this operation at the same time to obtain a desired surface. 
 This represents a step of changing the inverse image of the set of all real numbers not larger than a suitable value close to and smaller than a fixed singular value into the inverse image of the set of all real numbers not larger than a suitable value close to and larger than a fixed singular value. 

This gives a desired manifold and a function as FIGURE \ref{fig:4}. This completes the proof.

\begin{figure}
\includegraphics[width=100mm]{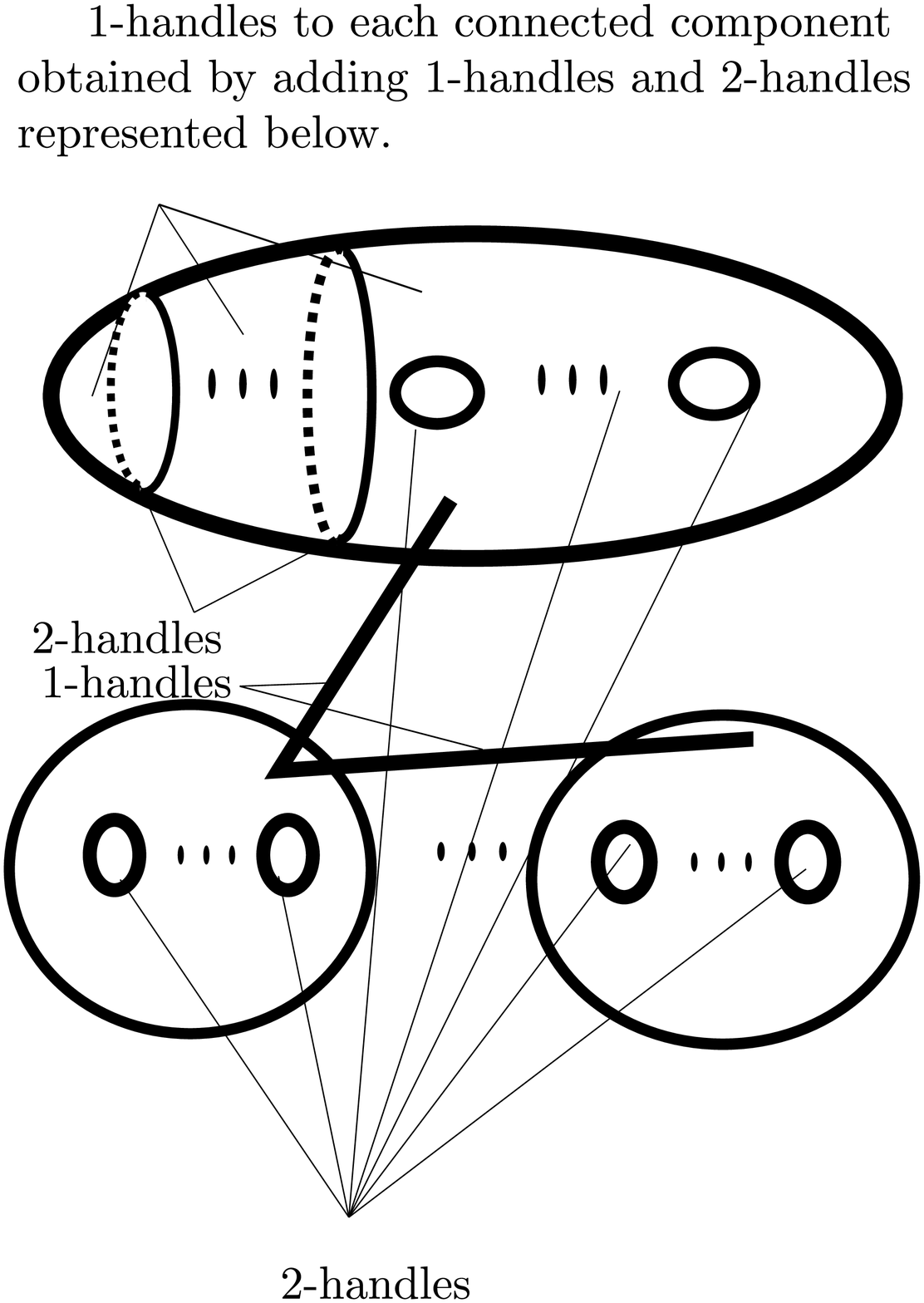}
\caption{Changing a surface ${\partial}_1 P$ into ${\partial}_2 P$ by adding handles at the same time. }
\label{fig:3}
\end{figure}
\begin{figure}
\includegraphics[width=35mm]{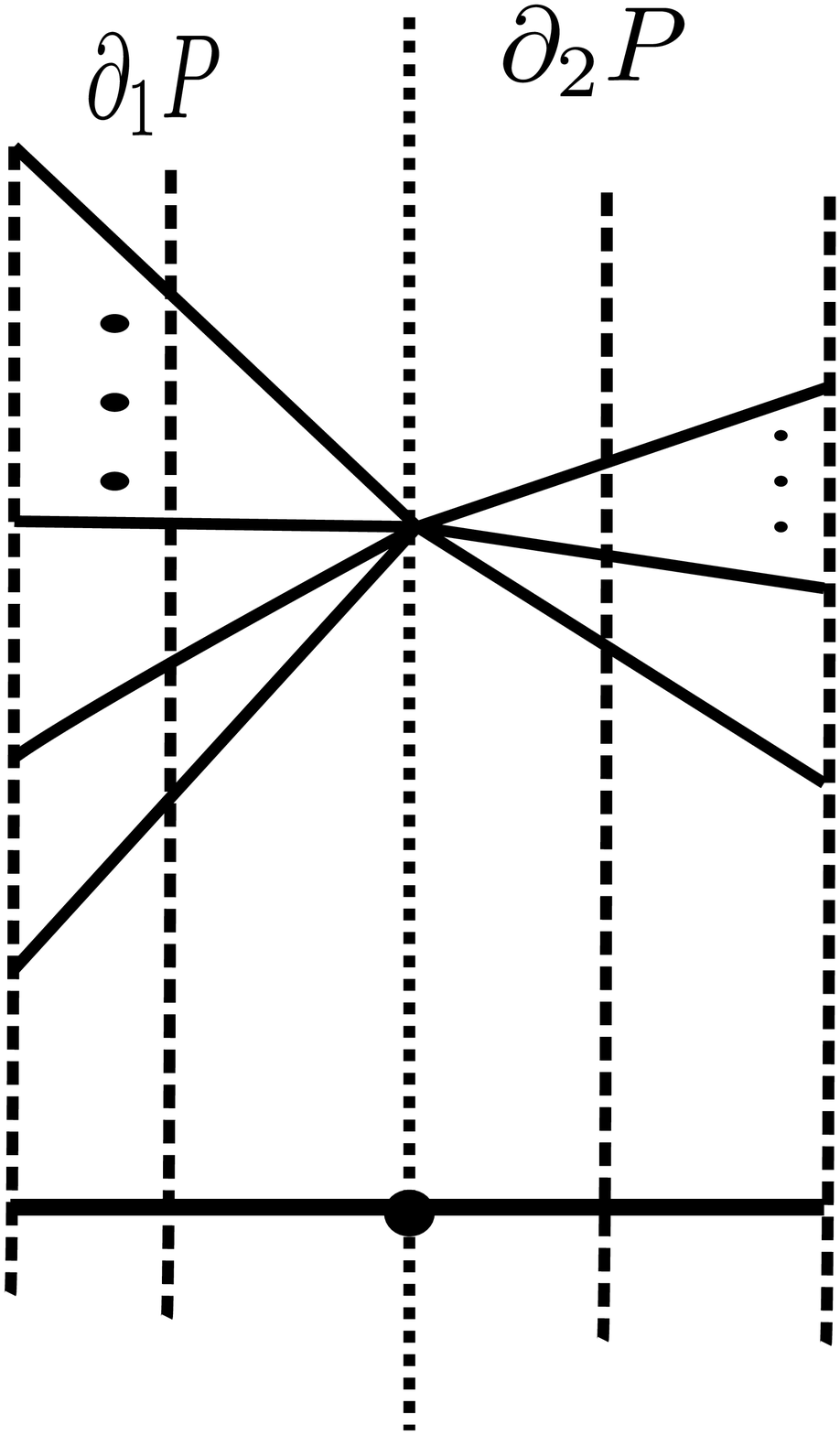}
\caption{A desired function on a desired compact manifold.}
\label{fig:4}
\end{figure}
\end{proof}

Note that related with the discussion, in \cite{masumotosaeki}, local construction of functions which may not be (stable) Morse as this on a surface was demonstrated to construct functions globally. \cite{michalak2} is also closely related to these studies: homology groups
 of Reeb graphs of (Morse) functions on fixed manifolds have been investigated.

\begin{proof}[Proof of Theorem \ref{thm:1}]
We first show the statement in the case where $n=1$.\\
\\
STEP 1 \\
For each element of the generator of $A_2$, we construct a local function as in Proposition \ref{prop:1}. For any generator $a$, we can represent $a$ as $a_{+}+a_{-}$ where $a_{+}$ is a linear combination of elements representing oriented diffeomorphism types of closed, connected and oriented spheres such that each coefficient is positive and where $a_{-}$ is a linear combination of elements representing oriented diffeomorphism types of closed, connected and oriented spheres such that each coefficient is negative. For example, if $a$ represents the $2$-dimensional sphere $S^2$, then we represent this by $ka+(-(k-1)a)$ for a positive integer $k>1$. By virtue of Proposition \ref{prop:1}, we set ${\partial_1} P$ as a surface consisting of connected surfaces where there exist only surfaces such that the corresponding coefficients of the oriented diffeomorphism types in $a_{+}$ are not zero and where the coefficients indicate the numbers of the connected spheres, set ${\partial_2} P$ similarly and we can obtain a Morse
 function. For example, if $a$ represents the $2$-dimensional sphere $S^2$, then we can construct a function as presented in FIGURE \ref{fig:2} where $a_1$, $a_2$ and $a_3$ represent $S^2$. We consider a copy of each Morse function, glue each pair of functions on the boundaries canonically and we obtain a finite number of local Morse functions (FIGURE \ref{fig:5}). We extend each function to a surjective smooth map over $N$ by setting the map as a trivial bundle over the complement of the original image of the local function. \\

\begin{figure}
\includegraphics[width=55mm]{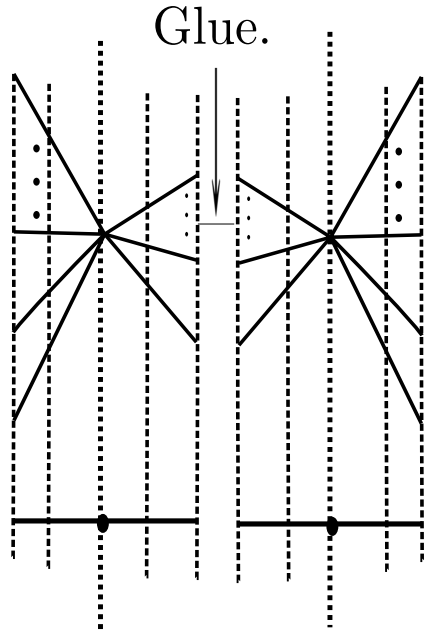}
\caption{Local Morse functions obtained by gluing two copies of local Morse functions as in Proposition \ref{prop:1}.}
\label{fig:5}
\end{figure}

\noindent STEP 2 \\
We add trivial bundles over $N$ such that fibers are closed and connected oriented surfaces and that the oriented diffeomorphism types of the surfaces are not in the module $A_1$ if we need. We need to construct a map so that as inverse images of regular values, every closed and connected orientable surface whose corresponding element is not in $A_1$ must appear.\\
\\ 
STEP 3 \\
Take a closed interval in the regular value set and its inverse image. We can change the local function there into a function as in Proposition \ref{prop:1} and this makes the source manifold of the global map over a connected manifold $N$ connected. See also
 FIGURE \ref{fig:6}. Note also that the fact $N$ is closed makes $M$ closed.
\begin{figure}
\includegraphics[width=55mm]{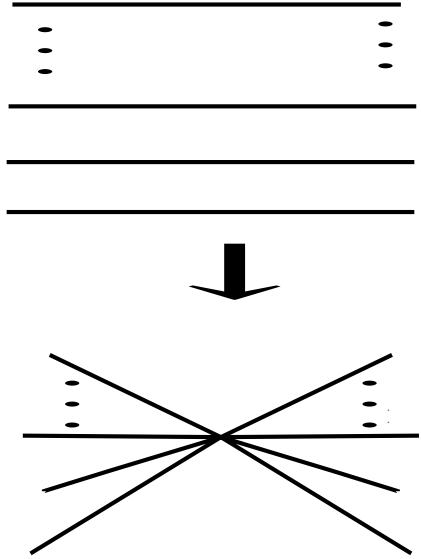}
\caption{An operation in STEP 3.}
\label{fig:6}
\end{figure}

\ \\
For the case where $n>1$, in STEP 1, we need to consider the product of the identity map on $S^{n-1}$ and each function, which is a fold map as a result, so that one of the region of the complement of the image, being a small closed tubular neighborhood of the singular value set, is a standard closed disc of dimension $n$ and we extend this by constructing trivial bundles over the complement of the original image of the local fold map. We can demonstrate operations in STEP 2 and STEP 3 similarly. See also FIGURE \ref{fig:7}

\begin{figure}
\includegraphics[width=55mm]{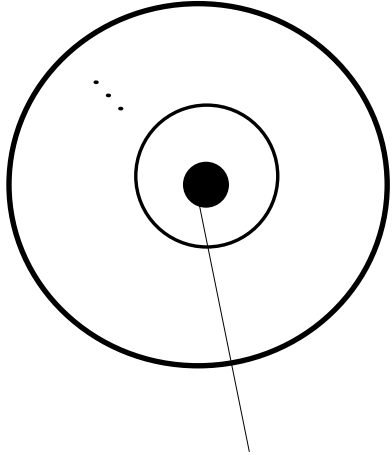}
\caption{The case where $n>1$ holds.}
\label{fig:7}
\end{figure}

\end{proof}
\begin{Rem}
We cannot show Theorem \ref{thm:1} in the case of ${\mathcal{N}}_2(R)$ or the version without "oriented" similarly. The obstruction to this case is the fact that ${\mathbb{R}}P^2$ is not null-cobordant, for example.
\end{Rem}


\begin{thebibliography}{25} 
\bibitem{golubitskyguillemin} M. Golubitsky and V. Guillemin, \textsl{Stable Mappings and Their Singularities}, Graduate Texts in Mathematics (14), Springer-Verlag(1974).
\bibitem{hiratukasaeki} J. T. Hiratuka and O. Saeki, \textsl{Triangulating Stein factorizations of generic maps and Euler Characteristic formulas}, RIMS Kokyuroku Bessatsu B38 (2013), 61--89. 
\bibitem{hiratukasaeki2} J. T. Hiratuka and O. Saeki, \textsl{Connected components of regular fibers of differentiable maps}, in "Topics on Real and Complex Singularities", Proceedings of the 4th Japanese-Australian Workshop (JARCS4), Kobe 2011,  World Scientific, 2014, 61--73. 
\bibitem{kitazawa} N. Kitazawa, \textsl{Smooth maps compatible with simplicialstructures and inverse images}, submitted to a refereed journal, arxiv:1802.06381.
\bibitem{masumotosaeki} Y. Masumoto and O. Saeki, \textsl{A smooth function on a manifold with given Reeb graph}, Kyushu J. Math. 65 (2011), 75--84.
\bibitem{michalak} L. P. Michalak, \textsl{Realization of a graph as the Reeb graph of a Morse function on a manifold}. to appear in Topol. Methods Nonlinear Anal., Advance publication (2018), 14pp, arxiv:1805.06727.
\bibitem{michalak2} L. P. Michalak, \textsl{Combinatorial modifications of Reeb graphs and the realization problem}, arxiv:1811.08031.
\bibitem{milnor} J. Milnor, \textsl{Lectures on the h-cobordism theorem}, Math. Notes, Princeton Univ. Press, Princeton, N.J. 1965.
\bibitem{morin} B. Morin, \textsl{Formes canoniques des singulariti\'{e}s d\'{}une application diff\'{e}rentiable}, C. E. Acad. Sci. Paris 260 (1965), 5662--5665, 6503--6506.
\bibitem{reeb} G. Reeb, \textsl{Sur les points singuliers d\'{}une forme de Pfaff compl\'{e}tement int\`{e}grable ou d\'{}une fonction num\'{e}rique}, Comptes Rendus
 Hebdomadaires des S\'{e}ances de I\'{}Acad\'{e}mie des Sciences 222 (1946), 847--849.
\bibitem{saeki} O. Saeki, \textsl{Notes on the topology of folds}, J.Math.Soc.Japan Volume 44, Number 3 (1992), 551--566.
\bibitem{sharko} V. Sharko, \textsl{About Kronrod-Reeb graph of a function on a manifold}, Methods of Functional Analysis and
 Topology 12 (2006), 389--396.
\bibitem{shiota} M. Shiota, \textsl{Thom's conjecture on triangulations of maps}, Topology 39 (2000), 383--399. 

\end{thebibliography}
\end{document}